\documentclass[12pt]{amsart}
\usepackage{amscd,amsmath,amsthm,amssymb}
\usepackage{pgf,tikz}
\usetikzlibrary{arrows}

\numberwithin{equation}{section}

\def\frk{\frak}               

\def\Phi{{\frk n}}
\def\Phi{{\frk N}}
\def\opn#1#2{\def#1{\operatorname{#2}}} 
\opn\chara{char} \opn\length{\ell} \opn\pd{pd} \opn\rk{rk}
\opn\projdim{proj\,dim} \opn\injdim{inj\,dim} \opn\rank{rank}
\opn\depth{depth} \opn\grade{grade} \opn\height{height}
\opn\embdim{emb\,dim} \opn\codim{codim}

\opn\Tr{Tr} \opn\bigrank{big\,rank}
\opn\superheight{superheight}\opn\lcm{lcm}
\opn\trdeg{tr\,deg}
\opn\reg{reg} \opn\lreg{lreg} \opn\ini{in} \opn\lpd{lpd}
\opn\size{size}\opn\bigsize{bigsize}
\opn\cosize{cosize}\opn\bigcosize{bigcosize}
\opn\sdepth{sdepth}\opn\sreg{sreg}
\opn\link{link}\opn\fdepth{fdepth}
\opn\index{index}
\opn\index{index}
\opn\indeg{indeg}
\opn\N{N}
\opn\SSC{SSC}
\opn\SC{SC}
\opn\lk{lk}
\opn\div{div} \opn\Div{Div} \opn\cl{cl} \opn\Cl{Cl}
\opn\Spec{Spec} \opn\Supp{Supp} \opn\supp{supp} \opn\Sing{Sing}
\opn\Ass{Ass} \opn\Min{Min}\opn\Mon{Mon} \opn\dstab{dstab} \opn\astab{astab}
\opn\Syz{Syz}
\opn\reg{reg}
\opn\Ann{Ann} \opn\Rad{Rad} \opn\Soc{Soc}
\opn\Im{Im} \opn\Ker{Ker} \opn\Coker{Coker} \opn\Am{Am}
\opn\Hom{Hom} \opn\Tor{Tor} \opn\Ext{Ext} \opn\End{End}\opn\Der{Der}
\opn\Aut{Aut} \opn\id{id}

\opn\nat{nat}
\opn\pff{pf}
\opn\Pf{Pf} \opn\GL{GL} \opn\SL{SL} \opn\mod{mod} \opn\ord{ord}
\opn\Gin{Gin} \opn\Hilb{Hilb}\opn\sort{sort}
\opn\initial{init}
\opn\ende{end}
\opn\height{height}
\opn\type{type}
\opn\aff{aff} \opn\con{conv} \opn\relint{relint} \opn\st{st}
\opn\lk{lk} \opn\cn{cn} \opn\core{core} \opn\vol{vol}
\opn\link{link} \opn\Link{Link}\opn\lex{lex}
\opn\gr{gr}

\def\pot#1#2{#1[\kern-0.28ex[#2]\kern-0.28ex]}

\opn\dirlim{\underrightarrow{\lim}}
\opn\inivlim{\underleftarrow{\lim}}
\def\Implies{\ifmmode\Longrightarrow \else
        \unskip${}\Longrightarrow{}$\ignorespaces\fi}
\def\implies{\ifmmode\Rightarrow \else
        \unskip${}\Rightarrow{}$\ignorespaces\fi}
\def\iff{\ifmmode\Longleftrightarrow \else
        \unskip${}\Longleftrightarrow{}$\ignorespaces\fi}

\let\:=\colon
\newtheorem{Theorem}{Theorem}[section]
 \newtheorem{Lemma}[Theorem]{Lemma}
 \newtheorem{Corollary}[Theorem]{Corollary}
 \newtheorem{Proposition}[Theorem]{Proposition}
 \newtheorem{Remark}[Theorem]{Remark}
 
 \newtheorem{Example}[Theorem]{Example}
 
 \newtheorem{Definition}[Theorem]{Definition}

\newtheorem{Notation}[Theorem]{Notation}

\let\epsilon\varepsilon
\let\kappa=\varkappa
\def\qed{\ifhmode\textqed\fi
      \ifmmode\ifinner\quad\qedsymbol\else\dispqed\fi\fi}
\def\textqed{\unskip\nobreak\penalty50
       \hskip2em\hbox{}\nobreak\hfil\qedsymbol
       \parfillskip=0pt \finalhyphendemerits=0}
\def\dispqed{\rlap{\qquad\qedsymbol}}

\opn\dis{dis}
\def\pnt{{\raise0.5mm\hbox{\large\bf.}}}

\opn\Lex{Lex}

\begin{document}

\title{The $a_0$-invariants of  powers of a two-dimensional squarefree monomial ideal}

\author{Lizhong Chu,  Dancheng Lu$^*$}

\thanks{* Corresponding author. }

\address{ School  of Mathematical Sciences, Soochow University, 215006 Suzhou, P.R.China}
\email{chulizhong@suda.edu.cn, ludancheng@suda.edu.cn}



\begin{abstract} Let $\Delta$ be an one-dimensional simplicial complex on $\{1,2,\ldots,s\}$ and $S$ the polynomial ring $K[x_1,\ldots,x_s]$ over a field $K$.  The explicit formula for $a_0(S/I_{\Delta}^n)$ is presented when $\mathrm{girth}(\Delta)\geq 4$. If $\mathrm{girth}(\Delta)=3$ we characterize the simplicial complexes $\Delta$ for which  $a_0(S/I_{\Delta}^n)=3n-1$ or $3n-2$.
\end{abstract}
\maketitle
\section*{Introduction}

Throughout  this paper except  subsection 1.2, we always assume that   $\Delta$ is   an one-dimensional simplicial complex on $[s]$ without isolated vertices, where $s$ is a positive integer and $[s]:=\{1,2,\ldots,s\}$.  Let $S:=K[x_1,\ldots,x_s]$ be the polynomial ring over a field $K$. The Stanley-Reisner ideal $I_{\Delta}$ of $\Delta$ is a pure two-dimensional squarefree  monomial ideal and vice versa. Here an ideal is {\it pure} if  all its minimal ideals have the same height. The algebraic properties of (symbolic) powers of the Stanley-Reisner ideal $I_{\Delta}$ of $\Delta$ were studied in \cite{HT, L,MV}. It was in \cite{HT} that the regularity of $S/I_{\Delta}^{(n)}$ is computed and in \cite{L} the geometric regularity of $S/I_{\Delta}^{n}$ was obtained and   the following formula was given:
$$\mbox{g-reg}(S/I_{\Delta}^n)=\mbox{reg}(S/I_{\Delta}^{(n)}).$$  Recently Minh and Vu in \cite{MV} extend the formula above a big step by proving the following theorem.

\noindent {\bf  Theorem} {\em Let $I$ be an intermediate monomial ideal between $I_{\Delta}^n$ and $I_{\Delta}^{(n)}$, namely $I=I_{\Delta}^n+(f_1,\ldots,f_t)$, where $f_i$ are among minimal generators of $I_{\Delta}^{(n)}$.  Then}  $$\mbox{reg}(S/I)= \mbox{g-reg}(S/I_{\Delta}^n)=\mbox{reg}(S/I_{\Delta}^{(n)}).$$

 All these results depend heavily  on the computing of $a_i$-invariants of the corresponding monomial ideals by Takayama's Lemma. Unlike  the regularity, the $a_i$-invariants of $S/I_{\Delta}^n$ and $S/I_{\Delta}^{(n)}$ may be different. For example, it was shown  in  \cite{L} that   $a_2(S/I_{\Delta}^n)$ is always equivalent to  $a_2(S/I_{\Delta}^{(n)})$, however  $a_1(S/I_{\Delta}^n)$ may be different from $a_1(S/I_{\Delta}^{(n)})$. Regarding  $a_0$-invariants, we know well that $a_0(S/I_{\Delta}^{(n)})$ always vanishes, but the computing of $a_0(S/I_{\Delta}^{n})$ seems very difficult. This is why the geometric regularity   instead of the regularity of $S/I_{\Delta}^n$  was obtained in \cite{L}.  In \cite{MV}, although the authors presented the regularity of $S/I_{\Delta}^n$, they didn't give
 $a_0(S/I_{\Delta}^n)$ explicitly. In fact, they obtained the lower bound of $\mbox{reg}(S/I_{\Delta}^n)$ in other ways, say, by the known fact $\mathrm{reg}(M)\geq d(M)$, where $M$ is any finitely generated graded $S$-module and $d(M)$ is the largest degree of minimal generators of $M$.

 In this paper, we devote to  evaluating $a_0(S/I_{\Delta}^n)$. Note that since $\Delta$ is an  one-dimensional simplical complex, we may look it  upon as a simple graph. If $\mbox{girth}(\Delta)\geq 4$, we first show that $a_0(S/I_{\Delta}^n)\in \{-\infty, 2n-1,2n\}$ and next determine when $a_0(S/I_{\Delta}^n)\neq -\infty$. Finally, it is shown that if $a_0(S/I_{\Delta}^n)\neq -\infty$, then $a_0(S/I_{\Delta}^n)=2n-1$ if and only if $G$, the complementary graph of $\Delta$, contains no two odd cycles which are adjacent. For the case that  $\mbox{girth}(\Delta)=3$ we don't obtain a complete answer. We show that if  $\mbox{girth}(\Delta)=3$, then $a_0(S/I_{\Delta}^n)\leq 3n-1$  and then characterize the simplicial complexes $\Delta$ such that $a_0(S/I_{\Delta}^n)=3n-1$ or $a_0(S/I_{\Delta}^n)=3n-2$.

It is worth  pointing out that  all  characterization  results stated above have the same structure in essence.  More precisely,  they may be written as:

\vspace{2mm}

\noindent {\bf Proposition} {\it For $n\geq 2$, $a_0(S/I_{\Delta}^n)=f(n)$ if and only if $\Delta$ contains an induced subgraph which is isomorphic to a graph in $A_{f(n)}$ and contains no induced subgraphs isomorphic to a graph in $B_{f(n)}$.}
 \vspace{2mm}

Here, $f(n)\in \{2n-1, 2n, 3n-1, 3n-2\}$, and $A_{f(n)}, B_{f(n)}$ are classes of graphs. If $f(n)\in \{2n-1, 2n\}$, the proposition needs the assumption that $\mbox{girth}(\Delta)\geq 4$.  For examples:

(1) If $f(n)=3n-1$ then $A_{f(n)}=\{\mathbb{K}_4 \}$ and $B_{f(n)}$ is the empty set; Here $\mathbb{K}_i$ means a complete graph with  $i$ vertices.

(2) If $f(n)=3n-2$ then $A_{f(n)}=\{\mathbb{G}'_1, \mathbb{F}_2, \ldots, \mathbb{F}_{n-1} \}$ and $B_{f(n)}=A_{3n-1}$. Here the definitions of  $\mathbb{G}'_1, \mathbb{F}_k$ are given in Theorem~\ref{3n-2} and Notation~\ref{F_k}, respectively.

(3) If $f(n)=2n$  then $A_{f(n)}=\{ \overline{\mathbb{G}_3},\overline{\mathbb{G}_4}, \overline{\mathbb{G}_5 }\}$ and $B_{f(n)}=\{\mathbb{K}_3\}$; Here $\overline{G}$ denotes the  complementary graph of  a graph $G$.  The  graphs $\mathbb{G}_3, \mathbb{G}_4, \mathbb{G}_5$  are defined in Example~\ref{G3G4G5}.

(4) If $f(n)=2n-1$ then $A_{f(n)}=\emptyset$ and $B_{f(n)}=A_{2n}\cup B_{2n}$.
\vspace{2mm}

We may call a graph in $A_{f(n)}$ to be a {\em basic graph} for $f(n)$.  By Proposition~\ref{3n-3},
    both $\mathbb{G}_6$ and  $\mathbb{G}_7$  are basic graphs for  $3n-3$. However, it seems difficult  to obtain a complete list of basic graphs for $3n-3$.

\section{Preliminaries}

In this section, we fix notation and recall some concepts and results which will be used in this paper. For all undefined notation one may refer to \cite{HH}.

 \subsection{Graphs and edge ideals}
Let $G$ be a simple graph on vertex set $V(G)=[s]$ and with edge set $E(G)$.  For a vertex $i\in [s]$, the neighborhood of $i$ is the subset $\{j\in [s]\:\ \{i,j\}\in E(\Delta)\}$, which is denoted by $N_G(i)$ or simply $N(i)$.

 An {\it independent set} of $G$ is a subset $S\subseteq [s]$ such that $\{i,j\}\notin E(G)$ for all $i,j\in S$.  The {\it independent number} of $G$, denoted by $\alpha(G)$,  is the maximum size of independent sets of $G$.

 Given a subset $W$ of $[s]$ the induced subgraph of $G$ on $W$ is the subgraph $G_W$ on $W$ consisting of edges $\{i,j\}\in E(G)$ with $\{i,j\}\subseteq W$.

 A {\it walk} of $G$ of length $q$ between vertices $i$ and $j$, is a sequence of edges: $$\{i_0,i_1\},\{i_1,i_2\},\ldots, \{i_{q-1},i_q\},$$  with $i_0=i, i_q=j$,   which is often written as:  $i=i_0-i_1-i_2-\cdots-i_q=j$.  The walk is {\it closed} if $i=j$. A cycle $C$ of $G$ is a closed walk
$i_0-i_1-i_2-\cdots-i_q=i_0$ such that $i_{0},i_1,\ldots,i_{q-1}$ are pairwise distinct. The length of a cycle $C$ is denoted by $\ell(C)$. We may regard  the cycle $C$ as  a subgraph of $G$ with $V(C)=\{i_0,\ldots,i_{q-1}\}$ and $E(C)=\{\{i_k,i_{k+1}\}\:|k=1,2,\ldots, q-1 \}$. Thus an {\it induced cycle} is a cycle which is also an induced subgraph. The {\it girth} of $G$, denoted by $\mbox{girth}(G)$, is the smallest length of induced cycles. By convention $\mbox{girth}(G)=\infty$ if $G$ contains no cycles.

A {\it clique} of $G$ is a complete subgraph of $G$. The clique number  of  simple graph $G$, denoted by $\mathrm{clique}(G)$, is the largest   number of vertices of cliques of $G$.

The edge ideal of $G$ is defined to be the monomial ideal $$I(G)=(x_ix_j\:\ \{i,j\}\in E(G))$$ of $S$. By abuse of notation, we often identify  $i$ with $x_i$ for $i\in [s]$ and call $x_ix_j\in I(G)$ an edge of $G$.
\subsection{Stanley-Reisner ideal of a simplicial complex}
A simplicial complex $\Delta$ on $[s]$ is a collection of subsets of $[s]$ such that if $F_1\subseteq F_2\subseteq [s]$ and $F_2\in \Delta$, then $F_1\in \Delta$. The Stanley-Reisner ideal $I_{\Delta}$ is defined to the ideal $(x_{F}\:\ F\notin \Delta)$, where $x_F=\prod_{i\in F}x_i$.

An  one-dimensional simplicial complex is just a simple graph.  For convenience, if  $\Delta$ is an one-dimensional simplicial complex on $[s]$, we always assume further that $\Delta$ contains no isolated vertices, i.e., any vertex  of $\Delta$ belongs to an edge. In this case, $I_{\Delta}$ is a pure two-dimensional monomial ideal.

\subsection{$a_i$-invariants and $0$-th critical monomials} Let $\mathfrak{m}=(x_1,\ldots,x_s)$ be the maximal homogeneous ideal of $S$ and $M$ a finitely generated graded $S$-module.  For a vector $\mathbf{a}=(a_1,\ldots,a_s)\in \mathbb{Z}^s$, $|\mathbf{a}|$ denotes the number $a_1+\cdots+a_s$. The $a_i$-invariants of $M$ is defined to be $$a_i(M)= \left\{
                                                                                        \begin{array}{ll}
                                                                                         \max\{|\mathbf{a}|\:\ \mathbf{a}\in \mathbb{Z}^s, H_{\mathfrak{m}}^i(M)_{\mathbf{a}}\neq 0 \} , & \hbox{if $H_{\mathfrak{m}}^i(M)\neq 0 $;} \\
                                                                                          -\infty, & \hbox{otherwise.}
                                                                                        \end{array}
                                                                                      \right.
$$
Here, $ H_{\mathfrak{m}}^i(M)$ denotes the $i$-th local cohomology of $M$ with respect to $\mathfrak{m}$.
 For a monomial ideal $I$, Takayama found a formula for $\dim_K(H_{\mathfrak{m}}^i(S/I)_{\mathbf{a}})$ for all $\mathbf{a}\in \mathbb{Z}^{s}$ in terms of certain simplicial complexes which are called degree complexes and denoted by $\Delta_{\mathbf{a}}(I)$. By the fact $\Delta_{\mathbf{a}}(I)=\mathrm{Link}_{\Delta_{\mathbf{a_+}}(I)}(G_{\mathbf{a}})$ for any $\mathbf{a}\in \mathbb{Z}^s$, see e.g. \cite[Theorem 1.6]{MT}, Takayama's lemma can rewritten as follows.

\begin{Lemma}   Let $I$ be a  monomial ideal in the polynomial ring $S$. Then  \begin{equation*}\begin{split}
a_i(S/I)=\max\{|\mathbf{a}|-|F|\:\ \mathbf{a}\in \mathbb{N}^s ,H_{\mathfrak{m}}^i(S/I)_{\mathbf{a}}= \widetilde{H}_{i-|F|-1}(\mathrm{Link}_{\Delta_{\mathbf{a}}(I)}F; K)\neq 0, \\ \mbox{ where }  F\in \Delta(I) \mbox{ with } F\cap \mathrm{supp}(\mathbf{a})=\emptyset   \}.\end{split}\end{equation*}
\end{Lemma}

If $\mathbf{a}\in \mathbb{N}^s$, then $I_{\Delta_{\mathbf{a}}(I)}=\sqrt{I:x^{\mathbf{a}}}$, see e.g. \cite[Lemma 2.12]{MV}.

\begin{Lemma}  \label{fun} Let $I$ be a monomial ideal in $S$ and $\mathbf{a}$ a vector in $\mathbb{N}^s$. Then the following statements are equivalent:
\begin{enumerate}
\item $H_{\mathfrak{m}}^0(S/I)_{\mathbf{a}}\neq 0 $;

\item $\Delta_{\mathbf{a}}(I)$ is the empty complex $\{\emptyset\}$;

\item  $\sqrt{I:x^{\mathbf{a}}}=\mathfrak{m}$.
\end{enumerate}
\end{Lemma}

\begin{proof} We know from the theory of simplicial homology (see e.g. \cite[Page 80]{HH}), that $\widetilde{H}_{-1}(\Delta; K)\neq 0$ iff $\Delta=\{\emptyset\}$. Now the result is clear.
\end{proof}

\begin{Definition}\em  A vector $\mathbf{a}\in \mathbb{N}^n$ is called a {\it $0$-th  critical exponent} of $I$  if $\mathbf{a}$ satisfies one of the equivalent conditions in Lemma~\ref{fun}.  A  $0$-th critical exponent $\mathbf{a}$ of $I$ is called  {\it extremal} if
$|\mathbf{a}|=\max\{|\mathbf{b}|: \mathbf{b} \mbox{ is a $0$-th critical exponent of  }  I\}$. We call $\mathbf{u}=x^{\mathbf{a}}$ a {\it $0$-th  critical monomial } (resp.  {\it $0$-th  extremal monomial})  of $I$ if $\mathbf{a}$ is a $0$-th  critical exponent (resp. $0$-th  extremal exponent).

By this definition, if $\mathbf{u}$ is a $0$-th  extremal monomial of $I$, then  $a_0(S/I)=\deg(\mathbf{u})$.

\section{when $\mathrm{girth}(\Delta)\geq 4$}
 In this section, we always assume that $\Delta$ is  an one-dimensional simplicial complex on $[s]$ with $\mathrm{girth}(\Delta)\geq 4$. Looking upon $\Delta$ as a simple graph and let $G$ be the complementary graph of $\Delta$. Then $\alpha(G)=2$ and  it is not difficult to see that $I_{\Delta}$ is nothing but the edge ideal $I(G)$ of $G$. We will present a explicit formula for the $a_0$-invariant of $S/I_{\Delta}^n$.

We begin with recalling a result from \cite{CHL}, which is actually a variation of  \cite[Theorem~6.1 and
Theorem~6.7]{B}. For this, we use $\{\{\ldots\}\}$ to represent a {\it multiset}. In particular, $\{\{1,2\}\}\neq \{\{1,2,2\}\}$, $\{\{1,2,2\}\}\subseteq \{\{1,2,2,3,4\}\}$, but $\{\{1,2,2\}\}\nsubseteq \{\{1,2,3,4\}\}$, and $\{\{1,2,2,3,3\}\}\cap \{\{2,2,2, 3\}\}=\{\{2,2,3\}\}$.

\begin{Lemma}
\label{connectedness}   Let $G$ be a simple graph with edge ideal
$I=I(G)$. Let $m\geq 1$ be an integer and  let  $e_1, e_2, \ldots, e_m$ (maybe repeatedly) be edges  of $G$  and $\mathbf{v}\in S$ a monomial  such that
$e_1e_2\cdots e_m\mathbf{ v} \in I^{m+1}$.  Then there exist variables  $w$ and $y$
with $wy|\mathbf{v}$  and an odd walk in $G$ connecting $w$ to $y$:
$$ w=z_1-z_2-z_3-\cdots-z_{2t}-z_{2t+1}-z_{2t+2}=y$$
 such that $\{\{z_2z_3,\ldots,z_{2t}z_{2t+1}\}\}\subseteq  \{\{e_1,\ldots,e_m\}\}$. Here, if $t=0$, then the walk means the edge $w-y$ (i.e., the edge $wy$).
\end{Lemma}

If $\mathbf{u}$ is a monomial  and $I$ is a monomial ideal of $S$, as usual we set $\mathrm{order}_I(\mathbf{u})=\max\{n\geq 0\:\ \mathbf{u}\in I^n\}$. Following \cite{MV}, a {\it good decomposition} of $\mathbf{u}$ with respect to $I$ is a decomposition $\mathbf{u}=MN$ such that $M,N$ are monomials and $M$ is a minimal generator of $I^{\mathrm{order}_I(\mathbf{u})}$.

\begin{Proposition} \label{squarefree} Let $I$ be the edge ideal of a simple graph $G$.  If $\mathbf{u}$ is a $0$-th critical monomial of $I^n$, then for any good decomposition $\mathbf{u}=MN$ w.r.t. $I$, one has $\mathrm{order}_{I}(\mathbf{u})<n$ and $N$ is  squarefree.
\end{Proposition}

\begin{proof} If $\mathrm{order}_{I}(\mathbf{u})\geq n$, then $\mathbf{u}\in I^n$ and $\sqrt{I^n:\mathbf{u}}=S$, a contradiction. This implies $\mathrm{order}_{I}(\mathbf{u})< n$.  Set $\ell:=\mathrm{order}_{I}(x^{\mathbf{a}})$ and write  $M$ as $M=e_1\cdots e_{\ell}$ with $e_i\in E(G)$ for $i=1,\ldots,\ell$. To prove $N$ is  squarefree, we assume on the contrary that there exists  $i\in [r]$ with $\deg_i(N)\geq 2$. Since $M(x_i^kN)\in I^n\subseteq I^{\ell+1}$ for some $k>0$ ($\because \sqrt{I^n: \mathbf{u}}=\mathfrak{m}$), it follows from Lemma~\ref{connectedness} that there exist variables (vertices) $w$ and $y$ with $wy|x_i^kN$ and an odd walk $$ w=z_1-z_2-z_3-\cdots-z_{2t}-z_{2t+1}-z_{2t+2}=y$$ in $G$ such that $\{\{z_2z_3,\ldots,z_{2t}z_{2t+1}\}\}\subseteq  \{\{e_1,\ldots,e_{\ell}\}\}$. Note that  $wx_i|N$ and $yx_i| N$. Since $wy\nmid N$, it is clear that $w=y=x_i$ and so $x_i^2M\in I^{\ell+1}$. But  $x_i^2M$ divides $MN$, we have $\mathrm{order}_{I}(\mathbf{u})\geq \ell+1$, a contradiction.\end{proof}

In the rest part of this section, $G$ is always the the complementary graph of $\Delta$.

\begin{Proposition} \label{three}  Let $I$ be the Stanley-Reisner  ideal of an one-dimensional complex  $\Delta$ with $\mathrm{girth}(\Delta)\geq 4$. Then $a_0(S/I^n)\in \{-\infty,  2n-1, 2n\}.$
\end{Proposition}
\begin{proof}  Assume that  $a_0(S/I^n)\neq -\infty$. This means that the set of $0$-th extremal monomials of $I^n$ is not empty. Let $\mathbf{u}$ be  such an monomial  and  let $\mathbf{u}=MN$ be a good decomposition  with respect to $I$. Then $\mathrm{order}_I(M)\leq (n-1)$. Moreover, $N$ is squarefree by Proposition~\ref{squarefree}. It follows that $|N|\in \{0,1,2\}$ since $\alpha(G)\leq 2$.

We next show that $|N|\neq 0$ and $\mathrm{order}_I(M)= (n-1)$.
If $|N|=0$, then $\deg(x_1\mathbf{u})\leq (2n-1)$ and $x_1\mathbf{u}\notin I^n$. It follows that  $x_1\mathbf{u}$ is also a $0$-th critical monomial. This is impossible and  so $|N|\neq 0$. Similarly, if $\mathrm{order}_I(M)\leq (n-2)$, then $\deg (x_1\mathbf{u})\leq (2n-1)$  and $x_1\mathbf{u}\notin I^n$, and so  $x_1\mathbf{u}$ is also  $0$-th critical, a contradiction again.

Hence $|N|\in \{1,2\}$ and $\mathrm{order}_I(M)= (n-1)$. From this it follows that $\deg(\mathbf{u})\in \{2n-1, 2n\}$. Consequently, $a_0(S/I^n)\in \{2n-1, 2n\}.$
\end{proof}

The case when $G$ is disconnected is solved in the following example.

\begin{Example}\label{disconnected} \em Let $G$ be a disconnected simple graph with $\alpha(G)=2$. Then $G$ is the disjoint union of two complete graphs, say $\mathbb{K}_s$ and $\mathbb{K}_t$. (In this case, $\Delta$ is a complete bipartite graph). Here $\mathbb{K}_i$ denotes a complete graph with $i$ vertices. If either $s$ or $t$ is 2, then $a_0(S/I(G)^n)=-\infty$ for all $n$;  If both $s$ and $t$ are large than 2, then  $a_0(S/I(G)^n)$ equals to $-\infty$ for $n=1,2$, and to $2n$ for all $n\geq 3$ by using the same construction as in  the proof of  Proposition~\ref{2n}.
\end{Example}

We  study when $a_0(S/I_{\Delta}^n)=-\infty$.  Note that   $a_0(S/I_{\Delta}^n)=-\infty\Longleftrightarrow H^0_{\mathfrak{m}}(S/I_{\Delta}^n)=0$$\Longleftrightarrow \mathrm{depth}(S/I_{\Delta}^n)>0\Longleftrightarrow$ there exists no  $0$-th critical monomials of $I_{\Delta}^n$.

\begin{Lemma} \label{G1G2} Let $G$ be a connected simple graph with $\alpha(G)=2$. If there exists an induced subgraph of $G$ isomorphic to one of the  graphs in Figure~\ref{g4},
  then $\mathrm{depth}(S/I(G)^n)=0$ for $n\geq 3$.
\end{Lemma}

\begin{figure}[ht!]

\begin{tikzpicture}[line cap=round,line join=round,>=triangle 45,x=1.5cm,y=1.5cm]

\draw (7,1)-- (5,1)--(6,2)--(7,1);
\draw (6,2)--(6,3);

\draw (6,3) node[anchor=south east]{4};

\draw (6,2) node[anchor=south east]{3};

\draw (5,1) node[anchor=south east]{2};
\draw (7.2,1) node[anchor=south east]{1};

\fill [color=black] (6,3) circle (1.5pt);
\fill [color=black] (7,1) circle (1.5pt);\fill [color=black] (6,2) circle (1.5pt);
\fill [color=black] (5,1) circle (1.5pt);

\draw (11.7,0.5) node[anchor=north east]{$\mathbb{G}_2$};

\draw (12,1)-- (10,1)--(11,2)--(12,1);
\draw (11,2)--(11,3); \draw  (12,1)--(11,3);

\draw (11,3) node[anchor=south east]{4};

\draw (11,2) node[anchor=south east]{3};

\draw (10,1) node[anchor=south east]{2};
\draw (12.2,1) node[anchor=south east]{1};

\draw (6.3,0.5) node[anchor=north east]{$\mathbb{G}_1$};

\fill [color=black] (11,3) circle (1.5pt);
\fill [color=black] (12,1) circle (1.5pt);\fill [color=black] (11,2) circle (1.5pt);
\fill [color=black] (10,1) circle (1.5pt);

\end{tikzpicture}
\caption{}\label{g4}
\end{figure}

\begin{proof} Denote $I(G)$ by $I$.  For $n\geq 3$ let $\mathbf{u}=(x_1x_2)^{n-2}x_3^2x_4$. We show that $\mathbf{u}$ is a $0$-th critical monomial of $I^n$. It is clear that $\mathrm{order}_I(\mathbf{u})=n-1$ and $x_i\mathbf{u}\in I^n$ for any $i\leq 4$. If $i\geq 5$, then, since $\{x_4, x_2\}$ is an independent set of $G$, either $x_ix_4$ or $x_2x_i$ is an edge of $G$. It follows that $$x_i\mathbf{u}=(x_1x_2)^{n-3}(x_1x_3)(x_2x_i)(x_3x_4)=(x_1x_2)^{n-3}(x_1x_3)(x_2x_3)(x_ix_4)$$  belongs to $I^n$.  This shows that  that $\mathbf{u}$ is a $0$-th critical monomial of $I^n$, and hence $\mathrm{depth}(S/I^n)=0$ for $n\geq 3$.
\end{proof}

Recall that a triangle in a graph $G$ is {\it dominating} if every vertex of $G$ is adjacent to at least one  of vertices of this triangle. By \cite[Theorem 3.1]{HH1},  $G$ has a dominating triangle if and only if \begin{equation}\label{depthfunction1}
\tag{Formu.1}\mathrm{depth}(S/I(G)^n)=\left\{
                                                                                                \begin{array}{ll}
                                                                                                  1, & \hbox{$n=1$;} \\
                                                                                                  0, & \hbox{$n\geq 2$.}
                                                                                                \end{array}
                                                                                              \right.
\end{equation}

The case when $|V(G)|\leq 5$ is discussed in the following example.
\begin{Example} \em  Let $G$ be a connected simple graph with $\alpha(G)=2$ and $4\leq|V(G)|\leq 5$. If $G$ is a complete graph, then,
since $G$ contains a dominating triangle, the depth function $\mathrm{depth}(S/I(G)^n)$ of $I(G)$ is given in (\ref{depthfunction1}). Hence we now assume that $G$ is not a complete graph.

If $|V(G)|=4$, then $G\in \{\mathbb{G}_1,\mathbb{G}_2, \mathbb{P}_3, \mathbb{C}_4\}$.  Here $\mathbb{G}_1,\mathbb{G}_2$ are graphs in Lemma~\ref{g4},   $\mathbb{C}_4$ is a cycle of length 4, and $\mathbb{P}_3$ is a path of length 3. By   \cite[Theorem 4.4]{Tr}, $\mathrm{depth}(S/I(G)^n)$ is positive for all $n\geq 1$  if $G\in \{\mathbb{C}_4,\mathbb{P}_3\}$, while the depth functions of $I(G)$ with $G\in\{\mathbb{G}_1,\mathbb{G}_2\}$ are also given by (\ref{depthfunction1}).

Assume that $|V(G)|=5$. If $G=\mathbb{C}_5$, the cycle of length 5, then $\mathrm{depth}(S/I(G)^n)$ is positive for $n=1,2$ and is 0 for $n\geq 3$. If $G$ is not isomorphic to $\mathbb{C}_5$, then either $G$ contains a dominating triangle or $G$ is isomorphic to  the graph with edge set $E(G)=\{\{1,2\},\{2,3\},\{3,1\},\{1,4\},\{4,5\}\}$. In the former case, $\mathrm{depth}(S/I(G)^n)=0$ for $n\geq 2$; In the latter case
$\mathrm{depth}(S/I(G)^n)$ is positive if $n=2$ and is $0$ for $n\geq 3$. Moreover, in all the cases above, $a_0(S/I(G)^n)=2n-1$ if $\mathrm{depth}(S/I(G)^n)=0$.

\end{Example}

We consider the case when $|V(G)|\geq 6$.
\begin{Proposition}  Let $G$ be a connected simple graph with $\alpha(G)=2$ and $|V(G)|\geq 6$.

\begin{enumerate}
\item If $n\geq 3$, then $\mathrm{depth}(S/I(G)^n)=0$;

\item If $n=2$, then $\mathrm{depth}(S/I(G)^n)=0$ if and only if $G$ has a dominating triangle.

\end{enumerate}
\end{Proposition}
\begin{proof} If $G$ is a complete graph, then $\mathrm{depth}(S/I(G)^n)=0$ for $n\geq 2$; If $G$ is not a complete graph, then, since $G$ contains a triangle,  $G$ contains an induced subgraph isomorphic to either $\mathbb{G}_1$ or $\mathbb{G}_2$, and so $\mathrm{depth}(S/I(G)^n)=0$ for $n\geq 3$ by Lemma~\ref{G1G2}. This proves (1). As to   (2), it follows directly from \cite[Theorem 3.1]{HH1}.
\end{proof}

We now have known when $a_0(S/I_{\Delta}^n)\neq -\infty$. In the following propositions, we determine when $a_0(S/I_{\Delta}^n)= 2n-1$ or $2n$ in terms of the graphical properties of $G$.  We call a cycle $C_1$ is {\it  not adjacent to} another  cycle $C_2$  if $V(C_1)\cap V(C_2)=\emptyset$ and every vertex $C_1$ is not adjacent to any vertex in $C_2$. The following proposition is similar to \cite[Proposition 6]{CHL}.

\begin{Proposition}  \label{cycle} Let $I$ be the edge ideal of a simple graph $G$ and  $n\geq 2$ an integer. Assume that    any two odd cycles $C_1,C_2$  of $G$ with $\ell(C_1)+\ell(C_2)\leq n$ are adjacent to each  other.  Then, for any  $0$-th critical monomial $\mathbf{u}$  of $I^n$ and any good decomposition $\mathbf{u}=MN$ of $\mathbf{u}$, one has   $|N|\leq 1$.
\end{Proposition}

\begin{proof}  Assume on the contrary that $|N|\geq 2$. Then, by Proposition~\ref{squarefree}, there exists $x\neq y\in V(G)$ such that $xy$ divides $N$.  Write $M$ as $M=e_1e_2\cdots e_{m-1}$ for some $m\leq n$, where $e_i\in E(G)$ for all $i$.

 By a similar  argument as in Proposition~\ref{squarefree}, we see that
 there exist odd closed walks $C_1: x=x_1-x_2-x_3-\cdots-x_{2k}-x_{2k+1}-x_1=x$ and $C_2: y=y_1-y_2-\cdots -y_{2\ell}-y_{2\ell+1}-y_1=y$ such that
$$\{\{f_1,\ldots,f_{k}\}\}\subseteq \{\{e_1,\ldots,e_{m-1}\}\}$$ and $$\{\{g_1,\ldots,g_{\ell}\}\} \subseteq \{\{e_1,\ldots,e_{m-1}\}\},$$
where $f_i=x_{2i}x_{2i+1}$  for $i=1,\ldots,k$ and $g_i=y_{2i}y_{2i+1}$ for $i=1,\ldots,\ell$.

Denote by $F$ the multi-set $\{\{f_1,\ldots,f_{k}\}\}\cap \{\{g_1,\ldots,g_{\ell}\}\}$. We firstly  assume that  $F\neq \emptyset$. Then $$\mathbf{u}_1:=\frac{xyf_1\cdots f_k g_1\cdots g_{\ell}}{\prod_{e\in F}e} \mbox{\quad divides \quad} \mathbf{u},$$  and moreover, $$\mathrm{order}_I(\frac{\mathbf{u}}{\mathbf{u}_1})\geq m-1-(k+\ell-|F|).$$  We now show  $\mathrm{order}_I({\mathbf{u}_1})= (k+\ell-|F|)+1.$ Let $j$ be the minimal of $i$ with $f_i\in F$ and let $1\leq t\leq \ell$ such that $f_j=g_t$. Then either $x_{2j}=y_{2t+1}$ and $x_{2j+1}=y_{2t}$ or $x_{2j}=y_{2t}$ and $x_{2j+1}=y_{2t+1}$. Suppose that $x_{2j}=y_{2t+1}$ and $x_{2j+1}=y_{2t}$. In this case, $\mathbf{u}_1$ could be written as
 \begin{equation*}\begin{split}\mathbf{u}_1=\prod_{i=0}^{j-1}( x_{2i+1}x_{2i+2})\cdot (x_{2s+1}y_{2t-1})\cdot\qquad \qquad \qquad \\ \prod_{i=1}^{t-1}(y_{2i-1}y_{2i})\cdot  \prod_{i=t+1}^{\ell}(y_{2i}y_{2i+1})\cdot\frac{\prod_{i=s+1}^k f_i} {\prod_{e\in F\setminus \{\{f_j\}\}}e}.\end{split}\end{equation*}
By the choice of $j$, we see that $F\setminus \{\{f_j\}\} \subseteq \{\{f_{j+1},\cdots,f_k\}\}$. From this it follows that $\mathbf{u}_1$ is a product of  edges and $\mathrm{order}_I({\mathbf{u}_1})= (k+\ell-|F|)+1.$  From this it follows that $\mathrm{order}_I(\mathbf{u})\geq \mathrm{order}_I(\frac{\mathbf{u}}{\mathbf{u_1}})+ \mathrm{order}_I(\mathbf{u_1})\geq m$, a contradiction.   The case that $x_{2j}=y_{2t}$ and $x_{2j}=y_{2t+1}$  yields a contradiction similarly.

Secondly, we assume that  $F=\emptyset$. In this case we put $$\mathbf{u}_2:=xf_1\cdots f_{k}yg_1\cdots g_{\ell}=x_1x_2\cdots x_{2k+1}y_1y_2\cdots y_{2\ell+1}.$$
If $V(C_1)\cap V(C_2)=\emptyset$, there exist some vertex of $C_1$, say $x_1$, and some vertex of $C_2$, say $y_1$, such that $x_1y_1\in E(G)$,  since $C_1$ is adjacent to $C_2$. (Here, we use the observation that every odd closed walk contains an odd cycle). Thus,
$$\mathbf{u}_2=(x_1y_1) (x_2x_3) \cdots (x_{2k}x_{2k+1}) (y_2y_3) \cdots (y_{2k}y_{2k+1})\in I^{k+\ell+1}.$$
If $V(C_1)\cap V(C_2)\neq \emptyset$, we may harmlessly assume that $x_1=y_1$. Then
$$\mathbf{u}_2=(x_1y_{2\ell +1}) (x_2x_3) \cdots (x_{2k}x_{2k+1}) (y_1y_2) \cdots (y_{2k-1}y_{2k})\in I^{k+\ell+1}.$$
In both cases, we have $\mathrm{order}_I(\mathbf{u}_2)\geq k+\ell +1$. Note that $\frac{\mathbf{u}}{\mathbf{u}_2}$ contains the product of $(m-1-k-\ell)$ edges, it follows that $\mathrm{order}_I(\mathbf{u})\geq \mathrm{order}_I(\mathbf{u}_2)+\mathrm{order}_I(\frac{\mathbf{u}}{\mathbf{u}_2})\geq m$. This is a contradiction again.
  \end{proof}

\begin{Corollary} Let $G$ be a simple graph with $\alpha(G)=2$ such that  any two odd cycles $C_1,C_2$ of $G$ with $\ell(C_1)+\ell(C_2)\leq n$ are adjacent. Denote $I(G)$ by $I$.  Then $a_0(S/I^n)\in \{-\infty, 2n-1\}.$
\end{Corollary}
\begin{proof} By Proposition~\ref{cycle}, $a_0(S/I^n)\neq 2n$. It follows that   $a_0(S/I^n)\in \{-\infty, 2n-1\}$ by Proposition~\ref{three}.
\end{proof}

 Proposition~\ref{cycle} holds for any  graph even without the restriction $\alpha(G)=2$. Under the assumption that $\alpha(G)=2$, the converse of this proposition is also true.

\begin{Proposition} \label{2n} Let $I$ be the edge ideal of a simple graph $G$ with $\alpha(G)=2$. If there are odd cycles $C_1, C_2$ in $G$ which are not adjacent, then for any $n\geq 3$, there exists a $0$-th critical monomial $\mathbf{u}$ of $I^n$ such that  $\deg(\mathbf{u})=2n$.
\end{Proposition}
\begin{proof} For every odd cycle $C$ there exists an induced odd cycle  $C'$ with $V(C')\subseteq V(C)$. By this fact we may assume that $C_1$ and $C_2$ are induced odd cycles, which are either of length 5 or of length 3, since  $\alpha(G)=2$. Note that if $C$ is a cycle of length 5, then every vertex not in $V(C)$ is adjacent to at least one vertex of $C$. Hence  $C_1$ and $C_2$ have to be triangles, say, with  $V(C_1)=\{1,2,3\}$  and $V(C_2)=\{4,5,6\}$.

 Put $\mathbf{u}:=x_1^{n-2}x_2^{n-2}x_3x_4x_5x_6$. Since there is no edges $x_ix_j$ with $i\in \{1,2,3\}$ and $j\in \{4,5,6\}$, it is easy to see that $\mathrm{order}_I(\mathbf{u})=n-1$. We claim that $\mathrm{order}_I(x_i\mathbf{u})=n$ for every  $i\in [s]$, and then $\mathbf{a}$ is a $0$-th critical vector of $I^n.$ In fact, if $i\in [6]$, say $i=1$, then $x_i\mathbf{u}=(x_1x_3)(x_1x_2)^{n-1}(x_4x_5)x_6\in I^n$; If  $i\in [s]\setminus [6]$, then $i$ is adjacent to at least one of vertices $3,4$, say $3$. From this it follows that   $x_i\mathbf{u}=(x_ix_3)(x_1x_2)^{n-1}(x_4x_5)x_6\in I^n$. This proves our claim and the result follows.
\end{proof}
The situation described in Proposition~\ref{2n} does exist as the following example shown.

\begin{Example} \label{G3G4G5}\em  Let $G$ be the  graph $\mathbb{G}_3$ in Figure~\ref{g2}. Then $\alpha(G)=2$  and  $C_1$ is not adjacent to $C_2$, where $C_1$ and $C_2$ are the triangles  on the vertices $1,2,3$, and  the vertices $4,5,6$, respectively. Conversely, every graph satisfying the conditions in Proposition~\ref{2n} contains $\mathbb{G}_3$ as a subgraph. More precisely, a connected graph $G$ with $\alpha(G)\leq 2$ has two non-adjacent odd cycles if and only if it contains an induced subgraph which is isomorphic to one of the graphs: $\mathbb{G}_3$, $\mathbb{G}_4$, $\mathbb{G}_5$, where $\mathbb{G}_4:=\mathbb{G}_3+\{2,7\}$, and  $\mathbb{G}_5:=\mathbb{G}_3+\{2,7\}+\{1,7\}$. To prove this, let $C_1,C_2$ be non-adjacent triangles of $G$, say,  with vertex sets $V(C_1)=\{1,2,3\}$ and  $V(C_2)=\{4,5,6\}$. Since $G$ is connected, there is a path connecting a vertex in $V(C_1)$ to a vertex in $V(C_2)$. We may assume that this path is $3=z_0-z_1-\cdots-z_k-z_{k+1}=4$. It is clear that $k\geq 1$, and, for $i\in [k]$, $z_i$ is adjacent to all vertices of either $C_1$ or $C_2$. We now use induction on $k$. If $k=1$, then the subgraph of $G$ induced on $[6]\cup\{z_1\}$ is  isomorphic to one graph in $\mathbb{G}_3, \mathbb{G}_4,\mathbb{G}_5$,  and thus it is what we require. Suppose that  $k>1$. If $z_k$ is adjacent to all vertices of $C_1$, then the subgraph of $G$ induced on $[6]\cup\{z_k\}$ meets our requirement.  If $z_k$ is adjacent to all vertices of $C_2$, then we consider the triangle $C_1$ and the triangle  on vertices $z_k, 5,6$, which we denote by $C_3$. Note that there is a path $3=z_0-z_1-\cdots-z_k$ connecting $C_1$ and $C_3$, the conclusion follows from the induction.

\begin{figure}[ht!]

\begin{tikzpicture}[line cap=round,line join=round,>=triangle 45,x=1.5cm,y=1.5cm]

\draw (3.7,4)--(3.7,6)--(4.7,5)--(5.4,5)--(6.1,5)--(7.1,6); \draw (6.1,5)--(7.1,4)--(7.1,6); \draw (3.7,4)--(4.7,5); \draw (5.4, 5)--(7.1, 4);  \draw (5.4, 5)--(7.1, 6);

\draw (3.9,4) node[anchor=north east]{1};
\draw (4.9,4.8) node[anchor=north east]{3};
\draw (4.1,6.2) node[anchor=north east]{2};

\draw (5.5,4.8) node[anchor=north east]{7};
\draw (6.1,5.2) node[anchor=north west]{4};
\draw (7.1,6.2) node[anchor=north west]{5};
\draw (7.1,4) node[anchor=north east]{6};

\fill [color=black] (3.7,6) circle (1.5pt);
\fill [color=black] (3.7,4) circle (1.5pt);
\fill [color=black] (4.7,5) circle (1.5pt);

\fill [color=black] (5.4,5) circle (1.5pt);
\fill [color=black] (6.1,5) circle (1.5pt);
\fill [color=black] (7.1,6) circle (1.5pt);

\fill [color=black] (7.1,4) circle (1.5pt);

\draw (3,5) node[anchor=north east]{$ \mathbb{G}_3:$};

\draw (11,6) node[anchor=north east]{$ \mathbb{G}_4:=\mathbb{G}_3+\{2,7\}$};

\draw (11,4.5) node[anchor=north east]{$ \mathbb{G}_5:=\mathbb{G}_4+\{1,7\}$};

\end{tikzpicture}
\caption{}\label{g2}
\end{figure}
\end{Example}

We summarize the previous results in the following theorem.

\begin{Theorem}  Let $I$ be the Stanley-Reisner ideal  of an one-dimensional simplicial complex $\Delta$ with $\mbox{girth}(\Delta)\geq 4$.
Denote by $G$ the complementary graph of $\Delta$. Assume that $G$ is connected.
\begin{enumerate}
  \item For $n=2$, $a_0(S/I^n)=3$ if  there is a dominating triangle in $G$ and it is $-\infty$ otherwise;

  \item  For $n\geq 3$, the following statements are equivalent.

{\em (a)} $a_0(S/I^n)=2n$;

{\em (b)} $G$ contains two odd cycles which are not adjacent;

{\em (c)} $G$ contains an induced subgraph which is isomorphic to one graph of $\{\mathbb{G}_3, \mathbb{G}_4, \mathbb{G}_5\}$.

\item  If the equivalent conditions in {\em (2)} fails, then $a_0(S/I^n)=2n-1$ for $n\geq 3$.
\end{enumerate}
\end{Theorem}
We note  that the case when $G$ is disconnected has been solved in Example~\ref{disconnected}.
\section{when $\mathrm{girth}(\Delta)=3$}

In this section we try to understand $a_0(S/I_{\Delta}^n)$ when $\mathrm{girth}(\Delta)=3$. We show that $a_0(S/I_{\Delta}^n)\leq 3n-1$ and characterize simplicial complexes $\Delta$ for which $a_0(S/I_{\Delta}^n)=3n-1$ or $3n-2$. Some examples $\Delta$ with $a_0(S/I_{\Delta}^n)=-\infty$ for all $n\geq 0$ are given.

\begin{Remark} \label{degree} \em  It is known from the proof of \cite[Theorem 1]{Ta} (see also \cite[Remark 2.9]{MV}) that if $\mathbf{a}$ is a $d$th critical vector of $I$ for some $d\geq 0$, then $a_i\leq \rho_i(I)$ for all $i\in [s]$.   Here $\rho_i(I)=\max\{\deg_i(\mathbf{u}): \mathbf{u} \mbox{ is a minimal monomial generator of }   I\}$.
\end{Remark}

\end{Definition}

Let $\mathbf{u}=x^{\mathbf{a}}$  be a monomial and $V$ a subset $V\subseteq [s]$. We define $\deg(\mathbf{u}):=a_1+\cdots+a_s$ and $\deg_V(\mathbf{u}):=\sum_{i\in V}a_i$. If $V=\{i\}$ for some $i\in [s]$, we write $\deg_i(\mathbf{u})$ for $\deg_V(\mathbf{u})(=a_i)$. As usual, $\mathrm{supp}(\mathbf{u}):=\{i\in [s]\:\; \deg_i(\mathbf{u})\neq 0\}$. For a monomial ideal $I$,  let $I_V$ denote the ideal $(\mathbf{u}\in G(I)\:\; \mathrm{supp}(\mathbf{u})\subseteq V)$. It is easy to see that if $\mathbf{u}$ is a monomial with  $\mathrm{supp}(\mathbf{u})\subseteq V$, $\mathbf{u}\in I^n$ if and only if $\mathbf{u}\in I_V^n$, for each $n\geq1$.

\begin{Proposition} \label{3.1}  Let $\Delta$ be an one-dimensional simplicial complex  with $\mathrm{girth}(\Delta)=3$ and $V(\Delta)\geq 4$. Then
 \begin{enumerate}
\item $a_0(S/I_{\Delta}^n)\leq 3n-1$;

\item  $a_0(S/I_{\Delta}^n)= 3n-1$ if and only if $\mathrm{clique}(\Delta)\geq 4$.
\end{enumerate}

\end{Proposition}

\begin{proof}

(1) Let  $x^{\mathbf{a}}$ be a $0$-th  extremal  monomial of $I^n$. Then $a_i\leq n-1$  for all $i\in [s]$ by Remark~\ref{degree}. If $|\mathbf{a}|\geq 3n$, then $x^{\mathbf{a}}\in I^n$ by \cite[Lemma 3.1]{MV}. This implies $a_0(S/I_{\Delta}^n)\leq 3n-1$.

(2)  If $\mathrm{clique}(\Delta)\geq 4$, we may assume the vertices $1,2,3,4$ forms a clique of $\Delta$. Set $\mathbf{u}=x_1^{n-1}x_2^{n-1}x_3^{n-1}x_4^2$. Then, since $I_{[4]}$ is generated in degree 3, $\mathbf{u}\notin I^{n}$. However, $x_i\mathbf{u}\in I^n$ for all $i\in [s]$. From this, it follows that $\mathbf{u}$ is a $0$-th critical vector of $I^n$ and so $a_0(S/I_{\Delta}^n)= 3n-1$.

For the proof of the converse, we assume on the contrary that there exists $\Delta$ such that  $a_0(S/I_{\Delta}^n)= 3n-1$ and $\mathrm{clique}(\Delta)\leq 3$.  Let $\mathbf{u}$ be a $0$-th extremal vector of $I^n$. Since $|\mathbf{u}|=3n-1$, it follows that $|\mathrm{supp}(\mathbf{u})|\geq 4$ by  Remark~\ref{degree}. Thus, there exist two vertices in $\mathrm{supp}(\mathbf{u})$, say 1,2,  such that $\{1,2\}$ is not an edge of $\Delta$. Now, set $\mathbf{v}:=\frac{\mathbf{u}}{x_1x_2}$. Then $\deg_i(\mathbf{v})\leq n-1$ for all $i\in[s]$ and so $\mathbf{v}\in I^{n-1}$  by \cite[Lemma 3.1]{MV}. Hence $\mathbf{u}\in I^n$, a contradiction.
\end{proof}

We present examples of simple complexes $\Delta$ for which $a_0(S/I_{\Delta}^n)=-\infty$ for all $n\geq 2$. These examples are necessary  in the proof of Theorem~\ref{3n-2}.

\begin{Example} \label{G_1} \em Let $\Delta$ be the graph $\mathbb{G}_1$ in Figure~\ref{g4}. We show that there exists no $0$-th critical vector of $I_{\Delta}^n$ and so $a_0(R/I_{\Delta}^n)=-\infty$ for all $n\geq 1$.  Assume on the contrary that $\mathbf{u}=x_1^{a_1}x_2^{a_2}x_3^{a_3}x_4^{a_4}$ is a $0$-th critical monomial of $I_{\Delta}^n$ for some $n>0$. Note that $I_{\Delta}=(x_1x_2x_3, x_2x_4,x_1x_4)$, we may assume further that $a_1\leq a_2$. If $a_1+a_2\leq a_4$, then $\mathrm{order}_I(\mathbf{u})=a_1+a_2=\mathrm{order}_I(\mathbf{u}x_4^t)$ for all $t\geq 0$, a contradiction. If $a_1+a_2> a_4$ and $a_2-a_1\geq a_4$, then $\mathrm{order}_I(\mathbf{u})=a_4+\min\{a_1,a_3\}=\mathrm{order}_I(\mathbf{u}x_2^t)$ for all $t>0$,  a contradiction again.

Finally, if  $a_2-a_1<a_4<a_1+a_2$, then $$\mathrm{order}_I(\mathbf{u})=\left\{
                                                                                          \begin{array}{ll}
                                                                                            a_4+\min\{a_1-k,a_3\}, & \hbox{$a_4-(a_2-a_1)=2k$;} \\
                                                                                            a_4+\min\{a_1-k-1,a_3\}, & \hbox{$a_4-(a_2-a_1)=2k+1$.}
                                                                                          \end{array}
                                                                                        \right.
$$
Suppose that $a_4-(a_2-a_1)=2k$.  This implies that if $a_1-k\leq a_3$, then $\mathrm{order}_I(\mathbf{u})=a_4+a_1-k=\mathrm{order}_I(\mathbf{u}x_3^t)$ for $t>0$, another contradiction. If $a_1-k> a_3$, then $\mathrm{order}_I(\mathbf{u})=a_4+a_3=\mathrm{order}_I(\mathbf{u}x_2^t)$ for $t>0$. This is also impossible. The case that $a_4-(a_2-a_1)=2k+1$ can be proved similarly.
\end{Example}

\begin{Example} \label{G_2} \em For $k\geq 2$, we let $\mathbb{E}_k$ denote the simple graph on vertex set $[k+2]$ and with edge set $\{\{k+1,k+2\},\{k+1,i\}, \{k+2,i\}: i=1,2,\cdots,k\}.$ Put $I:=I_{\mathbb{E}_k}$ for some $k\geq 2$. We claim that either  $\mathrm{order}_I(x^{\mathbf{a}})=\mathrm{order}_I(x_{k+1}^tx^{\mathbf{a}})$ or $\mathrm{order}_I(x^{\mathbf{a}})=\mathrm{order}_I(x_{k+2}^tx^{\mathbf{a}})$ for all $\mathbf{a}\in \mathbb{N}^{k+2}$ and all $t\geq 0$. In fact, if $\mathbf{a}=(a_1,\ldots,a_{k+2})$, then we may write $x^{\mathbf{a}}$ as $$x^{\mathbf{a}}=\prod_{i=1}^k(x_{k+1}x_{k+2}x_i)^{b_i}\prod_{1\leq i<j\leq k}(x_ix_j)^{c_{i,j}}N,$$ where \begin{equation} \label{1}\sum_{i=1}^k b_i\leq \min\{a_{k+1},a_{k+2}\},\qquad
b_i+\sum_{j\neq i}c_{i,j}\leq a_i, \forall i=1,\ldots,k,
\end{equation}.

Since $I=(x_{k+1}x_{k+2}x_i, x_ix_j\:\; 1\leq i<j\leq k)$, we have $$\mathrm{order}_I(x^{\mathbf{a}})=\max\{\sum_{i=1}^k b_i+\sum_{1\leq i<j\leq k}c_{i,j}\:\, b_i, c_{i,j} \mbox{ meet the inqualities  in (\ref{1})} \}.$$  Consequently,   if $a_{k+1}\leq a_{k+2}$, then $\mathrm{order}_I(x^{\mathbf{a}})=\mathrm{order}_I((x_{k+2})^tx^{\mathbf{a}})$ for all $t\geq 1$, as claimed. Hence $a_0(S/I^n)=-\infty$.
\end{Example}

The following notion  is also necessary  in Theorem~\ref{3n-2}.
\begin{Notation}\em \label{F_k} Let $k\geq 1$. We use $\mathbb{F}_k$ to denote the graph whose vertex set is $[3+k]$, and whose edge set is $\{\{i, j\}|i\in [k], j\in\{k+1,k+2,k+3\}\}\cup \{\{k+1,k+2\}\}.$ It is easy to see that $\mathbb{F}_1$ is isomorphic to the graph $\mathbb{G}_1$ in Figure~\ref{g4}.

\end{Notation}

\begin{Remark} \em  \label{part} Let $I$ be a monomial ideal of $S$ and $V$ a subset of $[s]$. We have defined $I_V$  to be the monomial ideal whose minimal generators are minimal generators of $I$ with their supports $\subseteq V$. In some cases, we regard $I_V$ as an ideals of $K[x_i\:\; i\in V]$. It is easy to see if $\mathbf{u}$ is a $0$-th critical monomial of $I$ with $\mathrm{supp}(\mathbf{u})\subseteq V$, then $\mathbf{u}$ is also a $0$-th critical monomial of $I_V$.
\end{Remark}

We are now ready to present a characterization of  an one-dimensional simplicial complex $\Delta$ such that  $a_0(S/I_{\Delta}^n)=3n-2$ for $n\geq 3$.

\begin{Theorem}\label{3n-2}   Let $\Delta$ be an one-dimensional simplicial complex  with $\mathrm{girth}(\Delta)=3$ and $|V(\Delta)|=s\geq 4$. Let $n\geq 3$.  Then  $a_0(S/I_{\Delta}^n)=3n-2$ if and only if  $\mathrm{clique}(\Delta)= 3$  and  $\Delta$ contains an induced subgraph which is isomorphic to one of the following graphs: $\mathbb{G}'_1$, $\mathbb{F}_k:k=2,\ldots,n-1$, where $\mathbb{G}'_1:=\mathbb{G}_1-\{3,4\}$. Note  the definition of $\mathbb{F}_k$ is given in Notation~\ref{F_k}.
  \end{Theorem}

\begin{proof} Denote $I_{\Delta}$ by $I$. Assume that $\mathrm{clique}(\Delta)= 3$.  If $\Delta$ contains $\mathbb{G}'_1$ as its  induced subgraph, then it is not difficult to show that $\mathbf{u}:=x_1^{n-1}x_2^{n-1}x_3^{n-1}x_4$ is a $0$-th critical monomial of $I^n$, and so $a_0(S/I_{\Delta}^n)=3n-2$ by Proposition~\ref{3.1}. We now assume that  $\mathbb{F}_k$ is an induced subgraph of $\Delta$  for some $k\in \{2,\ldots, n-1\}$.  we claim $$\mathbf{u}:=x_1^{a_1}x_2^{a_2}\cdots x_k^{a_k} x_{k+1}^{n-1}x_{k+2}^{n-1}x_{k+3}$$ is a $0$-th critical monomial of $I^n$, where $a_1+\cdots+a_k=n-1$ and $a_i\geq 1$ for each $i$. First, we show that $\mathrm{order}_I(\mathbf{u})\leq n-1$. In fact, if $\mathrm{order}_I(\mathbf{u})\geq n$, then we may write $\mathbf{u}$ as $\mathbf{u}=e_1\cdots e_nN$, where each $e_i$ is a minimal generator of $I$. Put $V:=\{1,2,\ldots,k\}\cup \{k+3\}$. Then $\deg_V(e_i)\geq 1$ for $i=1,\ldots,n$. From this it follows that $\deg_V(e_i)= 1$ for $i=1,\ldots,n$, since $\deg_V(\mathbf{u})=n$. Hence $e_i\in \{x_{k+1}x_{k+3}, x_{k+2}x_{k+3}, x_{k+1}x_{k+2}x_{j}, j\in [k]\}$. Note that $\deg_{k+3}(\mathbf{u})=1$,  we have either $\deg_{k+1}(\mathbf{u})=n$ or $\deg_{k+2}(\mathbf{u})=n$, a contradiction. This proves that $\mathrm{order}_I(\mathbf{u})\leq n-1$.

Next, we show that $x_j\mathbf{u}\in I^n$ for all $j\in [s]$. If $j=1$ then we may  write $x_j\mathbf{u}=(x_1x_2) \prod_{i=1}^{n-1}(x_{k+1}x_{k+2}x_{k_i})$ such that each $k_i\in V$, and so  $x_j\mathbf{u}\in I^n$. The proofs of the case that $j\in [k]$ are all the same. If $j\in \{k+1,k+2\}$, say $j=k+1$, then  $x_j\mathbf{u}=(x_{k+1}x_{k+3})\prod_{i=1}^{n-1}(x_{k+1}x_{k+2}x_{k_i})$ with each $k_i\in [k]$, and so  $x_j\mathbf{u}\in I^n$. If $j=k+3$, then
$x_j\mathbf{u}=(x_{k+1}x_{k+3})(x_{k+2}x_{k+3})x_1\prod_{i=1}^{n-2}(x_{k+1}x_{k+2}x_{k_i})$ with each $k_i\in [k]$ and so $x_j\mathbf{u}\in I^n$.
 If $j\notin [k+3]$, then, since  $x_jx_{k+1}x_1\in I$, we may write $x_j\mathbf{u}=(x_jx_{k+1}x_1)(x_{k+2}x_{k+3})\prod_{i=1}^{n-2}(x_{k+1}x_{k+2}x_{k_i})$ such that each $k_i\in [k]$, and so $x_j\mathbf{u}\in I^n$. Thus, we have shown $x_j\mathbf{u}\in I^n$ for all $j$. From this it follows that $\mathbf{u}$ is a $0$-th critical monomial of $I^n$   and so $a_0(S/I^n)=3n-2$.

Conversely, suppose  that $a_0(S/I^n)=3n-2$. It is clear that from Proposition~\ref{3.1} that $\mathrm{clique}(\Delta)= 3$. We now assume on the contrary that $\Delta$ contains no   induced subgraphs isomorphic to one of graphs $\{\mathbb{G}'_1,  \mathbb{F}_k, k=2,\ldots, n-1\}$.  We need to obtain a contradiction. To this end, let $\mathbf{u}$ be a $0$-th extremal monomial of $I^n$.  Then  $|\mathrm{supp}(\mathbf{u})|\geq 5$ by Examples~\ref{G_1} and \ref{G_2}. Let $\Delta_{\mathbf{u}}$   denote the subgraph of $\Delta$ induced on $\mathrm{supp}(\mathbf{u})$. It is clear that $I_{\Delta_{\mathbf{u}}}$ is nothing but the ideal $I_{\mathrm{supp}(\mathbf{u})}$.

 Set  $\delta=|\{i\in [s]: \deg_i(\mathbf{u})=n-1\}|$. Then $0\leq \delta\leq 2$. We consider the following cases:

{\it Case 1:}  The case that $\delta=0$.    Note that if the subgraph $\Delta_{\mathbf{u}}$ contains  two disjoint non-edges, (where a pair $\{i,j\}$ with $i\neq j$ is called a non-edge if $\{i,j\}$ is not an edge), say $\{1,2\}$ and $\{3,4\}$, then $\mathbf{u}=(x_1x_2)(x_3x_4)\mathbf{v}$ with $\deg(\mathbf{v})=3(n-2)$ and $\deg_i(\mathbf{v})\leq n-2$ for $i=1,2$ and so $\mathbf{u}\in I^n$, a contradiction. Hence,  any pair of  non-edges in $\Delta_{\mathbf{u}}$ intersects if they  exist. From this together with the fact $\mathrm{clique}(\Delta)=3$,  $\Delta_{\mathbf{u}}$ has to be  isomorphic to $\mathbb{E}_3$. This is also impossible by Example~\ref{G_2} and Remark~\ref{part}.

{\it Case 2:} The case that $\delta=1$. We may assume that $\deg_1(\mathbf{u})=n-1$. Put $X:=\mathrm{supp}(\mathbf{u})\setminus \{1\}$ and $N_X(1):=\{j\in X\:\, \{i,j\}\in E(\Delta)\}$. If $N_X(1)\subsetneq X$, then we can obtain a contradiction by using a similar argument as in the first case. Hence $N_X(1)=V$. This implies $\deg_X(e)\geq 2$ for any minimal generator $e$ of $I_{\Delta_{\mathbf{u}}}$. Since $\deg_X(x_1^k\mathbf{u})=2n-1$,  it follows that $\mathrm{order}_{I}(x_1^k\mathbf{u})\leq n-1$ for all $k\geq 1$, a contradiction again.

{\it Case 3:} The case that $\delta=2$. We may assume that $\deg_1(\mathbf{u})=\deg_2(\mathbf{u})=n-1$. Suppose first that $\{1,2\}\notin E(\Delta)$. We claim that  the  subgraph of $\Delta$ induced  on $Y:=\mathrm{supp}(\mathbf{u})\setminus \{1,2\}$ is a complete graph.  In fact, if there exist vertices in $Y$, say $3,4$, such that $\{3,4\}\notin E(\Delta)$, then  $\mathbf{u}=(x_1x_2)(x_3x_4)x_1^{n-2}x_2^{n-2}\mathbf{v}$ belongs to $I^n$, by \cite[Lemma 3.1]{MV}, a contradiction. This proves our claim.  From this together with the fact $\mathrm{clique}(\Delta)=3$ it follows that   $Y$ contains 3 vertices, and $Y\setminus N(i)$ is not empty for $i=1,2$. Here $N(i):=\{j\in [s]\:\, \{i,j\}\in E(\Delta)\}$.  Note that if
$|(Y\setminus N(1))\cup (Y\setminus N(2))|\geq 2$  then $\mathbf{u}\in I^n$, hence it contains exactly one vertex. This implies the  subgraph of $\Delta$ induced  on $\mathrm{supp}(\mathbf{u})$ is isomorphic to $\mathbb{E}_3$, which is impossible, by Example~\ref{G_2} and Remark~\ref{part}.

   Suppose next that $\{1,2\}$ is an edge of $\Delta$. We may write $\mathbf{u}=x_1^{n-1}x_2^{n-1}\mathbf{v}$ and denote by $Z$ the support of $\mathbf{v}$. We claim that if $1$ is adjacent to every vertex $Z$ then $\mathrm{order}_I(x_1^k\mathbf{u})=n-1$ for any $k\geq 1$. In fact, if  $\mathrm{order}_I(x_1^k\mathbf{u})\geq n$ for some $k\geq 1$, then we may write $x_1^k\mathbf{u}=e_1e_2\cdots e_nN$, where each $e_i$ is a minimal generator of $I$. It is not difficult to see that $\deg_{Z}(e_i)=1$ for $i=1,\ldots,n$, and so $e_i$ is either $x_1x_2x_{k_i}$ or $x_2x_{k_i}$ for some $k_i\in Z$ for all $i$. This implies $\deg_2(x_1^k\mathbf{u})\geq n$,  a contradiction. Hence $Z\setminus N(1)$ is not empty. Similarly, $Z\setminus N(2)$ is also not empty. Next, we show that $Z\setminus N(1)=Z\setminus N(2)$ is a singleton. In fact, if this is not true, then  there exist $i\neq j\in Z$ such that $i$ is not adjacent to  1 and $j$ is not adjacent to 2. This implies that $\mathbf{u}=(x_1x_i)(x_2x_j) x_1^{n-2}x_2^{n-2}\mathbf{v}'$ belongs to $I^n$, a contradiction. (Here we use the easy fact that $x_ix_jx_k\in I$ if $i,j,k$ are pairwise distinct.)
 Hence we may assume that $Z\setminus N(1)=Z\setminus N(2)=\{3\}$.

 Set $U:=Z\setminus \{3\}$. Then $U$ contains at least two vertices. Since $\mathrm{clique}(\Delta)=3$, every pair of vertices of $U$ are not adjacent, i.e.,  the induced graph of  $\Delta$ on $U$ is an empty graph. (Recall a graph is an {\it empty graph} if its edge set is empty.) On the other hand,  by the assumption that  $\Delta$ contains no induced subgraph isomorphic to $\mathbb{G}'_1$, we have $3$ is adjacent to every vertex in $U$. Hence $\Delta_{\mathbf{u}}$ is isomorphic to $\mathbb{F}_k$ with $k=|U|$. This is also  impossible by our assumption.
\end{proof}

  For $f(n)=3n-3$, we can  present two of its basic graphs. (See the end of Introduction for the definition of basic graph.)
\begin{Proposition}\label{3n-3} If $\Delta$ contains an induced subgraph isomorphic to one of the graph in Figure~\ref{g3}, then $a_0(S/I_{\Delta}^n)= 3n-3$ for $n\geq 3$.

\begin{figure}[ht!]

\begin{tikzpicture}[line cap=round,line join=round,>=triangle 45,x=1.5cm,y=1.5cm]

\draw (7,1)-- (5,1)--(6,2)--(7,1);
\draw (6,2)--(6,3);

\draw (6,3) node[anchor=south east]{4};

\draw (6,2) node[anchor=south east]{3};

\draw (5,1) node[anchor=south east]{2};
\draw (7.3,0.9) node[anchor=south east]{1};

\draw (7.2,2) node[anchor=south east]{5};
\draw (7,2)--(7,1);

\fill [color=black] (6,3) circle (1.5pt);
\fill [color=black] (7,1) circle (1.5pt);\fill [color=black] (6,2) circle (1.5pt);
\fill [color=black] (5,1) circle (1.5pt);

\fill [color=black] (7,2) circle (1.5pt);

\draw (11.7,0.5) node[anchor=north east]{$\mathbb{G}_7$};

\draw (12,1)-- (10,1)--(11,2)--(12,1);
\draw (11,2)--(11,3);

\draw (11,3) node[anchor=south east]{4};

\draw (11,2) node[anchor=south east]{3};

\draw (10,1) node[anchor=south east]{2};
\draw (12.3,1) node[anchor=south east]{1};

\draw (6.3,0.5) node[anchor=north east]{$\mathbb{G}_6$};

\draw (12.3,2) node[anchor=south east]{5};
\fill [color=black] (12,2) circle (1.5pt);

\fill [color=black] (11,3) circle (1.5pt);
\fill [color=black] (12,1) circle (1.5pt);\fill [color=black] (11,2) circle (1.5pt);
\fill [color=black] (10,1) circle (1.5pt);

\draw (12,1)--(12,2)--(11,3);

\end{tikzpicture}
\caption{}\label{g3}
\end{figure}
\end{Proposition}
\begin{proof} Set $\mathbf{u}=x_1^{n-2}x_2^{n-1}x_3^{n-2}x_4x_5$. Then it is not difficult to check that $\mathbf{u}$ is a $0$-th critical monomial of $I_{\Delta}^n$. This, together with Theorem~\ref{3n-2}, implies $a_0(S/I_{\Delta}^n)= 3n-3$ for $n\geq 3$
\end{proof}

We close this paper by giving a lower bound for $a_0(S/I_{\Delta}^n)$ when it is not $-\infty$.

\begin{Remark} \em  Let $\Delta$ be an one-dimensional simplicial complex and $n$ an integer $\geq 2$. If $a_0(S/I_{\Delta}^n)\neq -\infty$, then $a_0(S/I_{\Delta}^n)\geq n+1$.  
 
 In fact, since $a_0(S/I_{\Delta}^n)\neq -\infty$, we may assume $\mathbf{u}=x_1^{a_1}x_2^{a_2}\cdots x_s^{a_s}$ is a $0$-th critical monomial of $I_{\Delta}^n$. Assume further that $a_1\geq 1$. Then, since $x_1^k\mathbf{u}\in I^n$ for some $k>0$, we may write $x_1^k\mathbf{u}\in I^n$ as $$x_1^k\mathbf{\mathbf{u}}=e_1e_2\cdots e_nN, $$ where each $e_i$ is a minimal generator of $I$ and $N$ is monomial. Set $V:=[s]\setminus \{1\}$. Note that  $\deg_V(e_i)\geq 1$ for each $i$, we have $a_0(S/I_{\Delta}^n)\geq \deg(\mathbf{u})=a_1+\sum_{i=1}^n\deg_V(e_i)\geq n+1$.
\end{Remark}

{\bf \noindent Acknowledgment}
We thank the anonymous referee for his/her careful reading and useful comments.
This project is supported by NSFC (No. 11971338)

\end{document}